\documentclass[11pt]{article}
 \usepackage[latin1]{inputenc}
 \usepackage[dvips]{graphicx}
 \usepackage{amsmath}
 \usepackage{amsthm}
 \usepackage{amsfonts}
 \usepackage{amssymb}
 \usepackage{layout}
 \usepackage{verbatim}
 \usepackage{alltt}
 \usepackage{xypic}
 \input xy
 \xyoption{all}
 
\oddsidemargin - 1.5 pt
\newcounter{theorem}[section]

\renewcommand{\thetheorem}{\arabic{section}.\arabic{theorem}}

\newenvironment{theorem}{\par\medskip\noindent\refstepcounter{theorem}
\bgroup{\hspace*{-0.15 cm}\bf{Theorem}
\thetheorem.}\bgroup\it}{\egroup \egroup\par\medskip}

\newenvironment{lemma}{\par\medskip\noindent\refstepcounter{theorem}
\bgroup{\hspace*{-0.15 cm}\bf{Lemma} \thetheorem.}\bgroup\it}{\egroup
\egroup\par\medskip}

\newenvironment{proposition}{\par\medskip\noindent\refstepcounter{theorem}
\bgroup{\hspace*{-0.15 cm}\bf{Proposition}
\thetheorem.}\bgroup\it}{\egroup \egroup\par\medskip}

\newenvironment{corollary}{\par\medskip\noindent\refstepcounter{theorem}
\bgroup{\hspace*{-0.15 cm}\bf{Corollary}
\thetheorem.}\bgroup\it}{\egroup \egroup\par\medskip}

\newenvironment{definition}{\par\medskip\noindent\refstepcounter{theorem}
\bgroup{\hspace*{-0.15 cm}\bf{Definition}
\thetheorem.}\bgroup}{\egroup \egroup\par\medskip}

\newenvironment{example}{\par\medskip\noindent\refstepcounter{theorem}
\bgroup{\hspace*{-0.15 cm}\bf{Example} \thetheorem.}\bgroup}{\egroup
\egroup\par\medskip}

\newenvironment{remark}{\par\medskip\noindent\refstepcounter{theorem}
\bgroup{\hspace*{-0.15 cm}\bf{Remark} \thetheorem.}\bgroup}{\egroup
\egroup\par\medskip}

\newcommand{\C}{\mathbb C}

\newcommand{\R}{\mathbb R}
\newcommand{\Q}{\mathbb Q}
\newcommand{\Z}{\mathbb Z}
\newcommand{\F}{\mathbb F}

\def\O{{\cal O}}
\def\a{{\mathfrak a}}
\def\b{{\mathfrak b}}

\title{Weight Reduction for Mod $\ell$ Bianchi Modular Forms}  
\author{Mehmet Haluk \c{S}eng\"{u}n and Seyfi T\"urkelli}
\date{}
\begin{document}
\maketitle

\abstract{Let $K$ be an imaginary quadratic field with class number one. We prove that mod $\ell$, a system of Hecke eigenvalues occurring in the first cohomology group of some congruence subgroup $\Gamma$ of SL$_2(\O_K)$ can be realized, up to twist, in the first cohomology with trivial coefficients after increasing the level of $\Gamma$ by $(\ell)$.}

\footnotetext[1]{{\it 2000 Mathematics Subject Classification}. Primary ; Secondary }
\footnotetext[2]{{\it Key words and phrases.} Bianchi groups, imaginary quadratic fields, eigenvalue systems }


 \section{Motivation and Summary}
 Let $\bf{G}$ be a connected semisimple algebraic group defined over $\Q$. Let $K$ be a maximal compact subgroup of the group of 
real points $G=\textbf{G}(\R)$ of $\bf{G}$ and denote by $X=G/K$ the associated global Riemannian symmetric space. A torsion-free arithmetic subgroup $\Gamma$ of $G$ acts properly and freely on $X$. In this case, the the locally symmetric space $\Gamma \backslash X$ is an 
Eilenberg-MacLane space for $\Gamma$ and the cohomology of $\Gamma$ is equal to the cohomology of $\Gamma \backslash X$.
That is

$$H^*(\Gamma ,E) \simeq H^*(\Gamma \backslash X,\tilde E )$$ 

\noindent
where $E$ is a rational finite dimensional representation of $G$ over $\C$ and $\tilde E$ is the local system that $E$ induces on $\Gamma \backslash X$. A theorem of Franke \cite{fr} describes the cohomology spaces $H^*(\Gamma ,E)$ in terms of the automorphic forms attached to $\bf{G}$. If we take $\textbf{G}=\textbf{SL}_2$, then the Eichler-Shimura theorem \cite[Chapter 8]{sh} says that the automorphic forms that appear in the cohomology spaces $H^1(\Gamma ,E)$ are the classical modular forms. 

 Motivated by the above paragraph, we define a \textit{Bianchi modular form} over an imaginary quadratic field $K$ as an automorphic form attached to $G=\text{Res}_{K/\Q}(\textbf{SL}_2)$ that appears in some $H^1(\Gamma,E(\C))$ where $\Gamma$ is a congruence subgroup of SL$_2(\O_K)$ (the level) and $E(\C)$ is a rational finite-dimensional representation of GL$_2(\C)$ over $\C$ (the weight). 
 
 Harris-Soudry-Taylor, Taylor and Berger-Harcos \cite{hst, ta1, bh}, under some hypothesis, were able to attach compatible families of $\lambda$-adic Galois representations of $K$ to Bianchi modular forms in accordance with Langlands philosophy. In the reverse direction, it is natural to ask if mod $\ell$ Galois representations of $K$ arise from mod $\ell$ Bianchi modular forms. We define a \textit{mod $\ell$ Bianchi modular form} as a cohomology class in some $H^1(\Gamma,E \otimes \overline{\F}_{\ell})$ where $E$ is a rational finite-dimensional representation of GL$_2(\O_K / (\ell))$ over $\F_{\ell}$. Unlike the case of the classical modular forms, mod $\ell$ Bianchi modular forms are not merely reductions of the (char 0) Bianchi modular forms. This is due to the possible torsion in the cohomology with coefficients over $\O_K$, see Taylor's thesis \cite{ta2}. 

Elstrod-Grunewald-Mennicke \cite{egm} were the first investigators of the connection between mod $\ell$ Bianchi modular forms and mod $\ell$ Galois
representations of imaginary quadratic fields. In his paper \cite{fi}, Figueiredo considered an analogue of Serre's conjecture in this setting but he only considered mod $\ell$ Bianchi modular forms in cohomology spaces with trivial coefficients. Motivated by a result of Ash and Stevens \cite{as2} for the classical modular forms, he assumed that a Hecke eigenvalue system attached to a mod $\ell$ Bianchi modular form, after increasing the level, would be attached, up to twist, to another form with trivial weight. 

 In this paper, we prove that what Figueiredo assumed is true following the technique used by Ash and Stevens in \cite{as2}. Our main corollary is as follows

 \begin{corollary}
 Let $K$ be an imaginary quadratic field of class number one and $\O$ be its ring of integers. Let $\a$ be an ideal of $\O$ that is prime to the ideal $(\ell)$
 where $\ell$ is a rational prime that is split in $\O$. Let $\Phi$ be a Hecke eigenvalue system occuring in $H^1(\Gamma_1(\a),E)$ where
 $E$ is a finite dimensional $\F_{\ell}[\text{GL}_2(\O / (\ell))]$-module. Then $\Phi$ occurs in $H^1(\Gamma_1(\a \ell),\F_{\ell})$, up to twist.
 \end{corollary}

 As an immediate corollary of the above, we get 
 
 \begin{corollary}
 Mod $\ell$, there are only finitely many eigenvalue systems with fixed level.
 \end{corollary}

 Note that due to the possible existence of torsion in the second cohomology with integral coefficients, we cannot in general lift mod $\ell$ forms to characteristic 0. So this result does not immediately imply mod $\ell$ congruences between Bianchi modular forms.
 \vspace{.1 in}
 
 \noindent \textit{Acknowledgements} We thank Gebhard Boeckle and Gabor Wiese for the very helpful comments and discussions. The first author was partially supported by the SFB/TR 45 `Periods, Moduli Spaces and Arithmetic of Algebraic Varieties' of the German Research Foundation (DFG).

 \vspace{.1 in}
 \noindent \textit{Notation} Once and for all, fix a quadratic imaginary field $K$ of class number one and an ideal $\a$ of $\O=\O_K$. Also fix  a rational prime $\ell$ that is coprime to $\a$ and splits in $\O$  as $\ell = \lambda \bar{\lambda}$ . Let $\b$ be an arbitrary ideal.
We use the following notation: 
\[
\begin{tabular}{rl} 
  
 M$_2(\O):$ & matrices in GL$_2(K)$ with entries in $\O$\\

$ \Gamma_0(\b):$ & $\Bigl \lbrace \bigl ( \begin{smallmatrix} a & b  \\ c & d \\ \end{smallmatrix} \bigr ) \in \text{SL}_2(\O) : c \equiv 0 \mod \b \Bigr \rbrace $ \\
 
$ \Gamma_1(\b):$ & $\Bigl \lbrace \bigl ( \begin{smallmatrix} a & b  \\ c & d \\ \end{smallmatrix} \bigr ) \in \text{SL}_2(\O) : c \equiv d-1 \equiv 0 \mod \a \Bigr \rbrace $\\

$\Delta:$ & $\Bigl \{ \bigl ( \begin{smallmatrix} a & b  \\ c & d \\ \end{smallmatrix} \bigr ) \in \text{M}_2(\O) : c \equiv 0 \mod \a \Bigr \}$ \\

$\Delta(\b):$ & $\Bigl \{ \bigl ( \begin{smallmatrix} a & b  \\ c & d \\ \end{smallmatrix} \bigr ) \in \text{M}_2(\O): c \equiv 0 \mod \a\b \Bigr \}$ \\

$P(\b):$ & $\Bigl \{ g \in \Gamma_1(\a) : g \equiv \bigl ( \begin{smallmatrix} 1 & 0  \\ 0 & 1 \\ \end{smallmatrix} \bigr )  \mod \b \Bigr \}$ \\

$\Gamma:$ & $\Gamma_1(\a)$ \\

$\Gamma(\b):$ & $\Gamma_1(\a\cdot \b)$ \\

$\Gamma^0(\b):$ & $\Gamma_1(\a) \cap \Gamma_0(\b)$ \\

 \end{tabular}\]


\section{Hecke Operators on Cohomology}

 In this section, we describe the Hecke operators on the cohomology. Let $R$ be a ring and $\tilde{\alpha}=\bigl( \begin{smallmatrix} \alpha & 0  \\ 0 & 1 \\ \end{smallmatrix} \bigr)$ where $\alpha$ is a prime element of $\O$. We follow the standart notations and put $\Gamma_{\alpha} := \Gamma \cap \tilde{\alpha}^{-1}\Gamma \tilde{\alpha}$ and $\Gamma^{\alpha}:=\Gamma \cap \tilde{\alpha} \Gamma \tilde{\alpha}^{-1}$. 
 
 Let $V$ be a right $R[\text{M}_2(\O)]$-module. We define the Hecke operator $T_{\alpha}$ on the cohomology as the composition $$\xymatrix{ H^1(\Gamma,V) \ar[d]^{res} & H^1(\Gamma,V) \\ H^1(\Gamma_{\alpha} ,V) \ar[r]^{\hat{\alpha}} & H^1(\Gamma^{\alpha} ,V) \ar[u]^{cores} & }$$ where the map $\hat{\alpha}$ is defined by $$ c \mapsto (g \mapsto  c(\alpha^{-1} g \alpha)\cdot \alpha^{\iota}) $$ where $c$ is a cocycle in $H^1(\Gamma_\alpha ,V)$ and $\alpha^{\iota}= \text{det}(\alpha)\alpha^{-1}$.

 One can describe Hecke operators $T_\alpha$ explicitly: suppose $\Gamma \alpha \Gamma = \bigsqcup_{_{i=1}}^{^m} \gamma_i \Gamma $. Given  $g \in \Gamma$ and $\gamma_i$, there is a unique $\gamma_{j(i)}$ such that $\gamma_{j(i)}^{-1} g \gamma_i \in \Gamma $. Then $$(T_\alpha c)(g) = \sum_{1\leq i \leq m}  c( \gamma_{j(i)}^{-1} g \gamma_i )\cdot \gamma_i^{\iota} $$ for all cocycles $c$ in $H^1(\Gamma,V)$ and $g \in \Gamma$. We note that this formula agrees with the one given in \cite[p.194]{as1}. 

 We define the \textit{Hecke algebra} $\mathbb{H}$ as the subalgebra of the endomorphisms algebra of $H^1(\Gamma,V)$ that is generated by the $T_{\pi}$'s where $\pi$ is a prime. Note that $\mathbb{H}$ is a commutative algebra.

 The induced module $Ind(V)=Ind(\Gamma, \Gamma(\b),V)$ is the set of $\Gamma(\b)$-invariant maps from $\Gamma$ to $V$, that is $$ Ind(V)= \{ f: \Gamma \rightarrow V \ | \ f(gh)= f(g)\cdot h \text{ for all } h \in \Gamma(\b) \}.$$ Then $Ind(V)$ is a right $\Gamma$-module with the action $(f \cdot y)(x)=f(yx)$ for $x,y \in \Gamma$ and $f \in Ind(V)$.

 We can extend the $\Gamma$-action on $Ind(V)$ to a right $\Delta$-action in the following way. Let $\alpha \in \Delta$ and $f \in Ind(V)$ and $x \in \Gamma$, then there are $\beta \in \Delta(\b)$ and $y \in \Gamma$ such that $\alpha x =y  \beta $. We define $$( f \cdot \alpha )(x)=f(y)\cdot \beta .$$
 
 A key tool is Shapiro's lemma:
 
\begin{proposition}
 There is an isomorphism $$ \theta : H^1(\Gamma,Ind(V)) \rightarrow H^1(\Gamma(\b),V)$$ given by $f \mapsto f(I)$ for every cocycle $f$ in $H^1(\Gamma,Ind(V))$ where $I$ denotes the identity matrix. Moreover, the Hecke operators commute with the Shapiro map $\theta $.
\end{proposition}

The fact that the Hecke operators commute with the Shapiro isomorphism $\theta $ was proved in a more general setting in \cite{as1}. See also \cite{wi} for a proof in the case of PSL$_2(\Z)$ using the same construction as ours for the Hecke operators.

A \textit{system of eigenvalues} of $\mathbb{H}$ with values in a ring $R$ is a ring homomorphism $\Phi : \mathbb{H} \rightarrow R$. We say that
an eigenvalue system $\Phi$ occurs in the $R\mathbb{H}$-module $A$ if there is a nonzero element $a \in A$ such that $Ta=\Phi(T)a$ for all $T$ in $\mathbb{H}$.

The following lemma is proved in \cite[Lemma 2.1]{as1}.
\begin{lemma}\label{lemma-irred}
 Let $F$ be a field and $V$ be a $F\Delta$-module which is finite dimensional over $F$. If an eigenvalue system $\Phi : \mathbb{H} \rightarrow F$ occurs in $H^n(\Gamma,V)$, then $\Phi$ occurs in $H^n(\Gamma,W)$ for some irreducible 
$F\Delta$-subquotient $W$ of $V$.

\end{lemma}

Thus it is enough to investigate the cohomology with irreducible coefficient modules if we are only interested in the eigenvalue systems. In the next two sections, we discuss the irreducible $\F_{\ell}[\text{GL}_2(\O / (\ell))]$-modules.

\section{The Irreducible Modules } \label{irreducible}
 For a nonnegative integer $k$, we are interested in the right representation $\tilde{E}_k$ of $\textbf{GL}_2$ on $Sym^k(\textbf{A}^2)$ where $\textbf{A}^2$ is the affine plane. Another model of this representation is given as follows. Given a commutative a ring $R$, we have $E_k(R) \simeq R[x,y]_{k}$ where the latter is the space of homogeneous degree $k$ polynomials in two variables over $R$. Note that $\lbrace X^{k-i}Y^i : 0 \leq i \leq k \rbrace $ is an $R$-basis of $E_k(R)$. 

For a polynomial $P(X,Y)$ in $E_k(\O)$ and a matrix $\bigl( \begin{smallmatrix} a & b  \\ c & d \\ \end{smallmatrix} \bigr)$ in $\text{M}_2(\O)$, the above mentioned representation is defined as
 $$  \bigl ( P \cdot  \bigl( \begin{smallmatrix} a & b  \\ c & d \\ \end{smallmatrix} \bigr)\bigr ) \bigl( X,Y \bigr )
 = P \bigl ( \bigl( \begin{smallmatrix} a & b  \\ c & d \\ \end{smallmatrix} \bigr) \bigl( \begin{smallmatrix} X  \\ Y \\ \end{smallmatrix} \bigr)  \bigr )= P \bigl( aX+bY,cX+dY \bigr ).$$
 
 The quotients rings $\O/ \lambda$ and $\O / \bar{\lambda}$ are canonically isomorphic to $\F_{\ell}$. Then $\text{M}_2(\O)$ acts on $E_k(\F_{\ell})$ in two different ways: through reduction by $\lambda$ and by $\bar{\lambda}$. 

In this note, we are interested in the absolutely irreducible representations of $\text{GL}_2(\O / (\ell))$ over $\F_{\ell}$. Given nonnegative $a,r$, put
$$E^a_r(\F_{\ell}):= \text{det}^a \otimes_{\F_{\ell}} E_r(\F_{\ell}) $$

\noindent It follows from a result of Brauer and Nesbitt \cite{bn} that the absolutely irreducible representations of $\text{GL}_2(\O / (\ell))=\text{GL}_2(\O / \lambda) \times \text{GL}_2(\O / \bar{\lambda})$ over $\F_{\ell}$ are
$$E^{a,b}_{r,s}(\F_{\ell}):= E^a_r(\F_{\ell}) \otimes_{\F_{\ell}} E^b_s(\F_{\ell})   \ \ \ , \ \ \ 0 \leq r,s \leq \ell-1, \ \ 0 \leq a,b \leq l-2$$

These are $\text{M}_2(\O)$ modules as well: $\text{M}_2(\O)$ acts on the first module through reduction by $\lambda$ and on the
second through reduction by $\bar{\lambda}$. 
For the rest of the paper, we will work over $\F_{\ell}$. So we simply write $E^{a,b}_{r,s}$. Moreover, we write $E_{r,s}$ when $a=b=0$.

\vspace{.1 in}
Let $E$ be a $\F_{\ell}[\text{M}_2(\O)]$-module. Given $0 \leq a,b \leq l-2$, we mean by $$H^*(\Gamma,E)^{(a,b)}$$
the cohomology group $H^*(\Gamma,E)$ twisted as a Hecke module. More precisely, let $v$ be an element of $H^*(\Gamma,E)$. Denote it as $v'$ when viewed as an element of $H^*(\Gamma,E)^{(a,b)}$. Let $\tau_1,\tau_2$ be the reduction maps from $\O$ to $\F_{\ell}$ by $\lambda$ and $\bar{\lambda}$ respectively. Given a Hecke operator $T_{\pi}$, we have
$$T_{\pi}(v')=\tau_1(\pi)^a\tau_2(\pi)^b T_{\pi}(v)$$ 

As SL$_2(\O)$-modules $E_{r,s}^{a,b}$ is the same as $E_{r,s}$. The difference occurs when they are considered as Hecke modules. The following observation is immediate.
\begin{lemma}\label{lemma-det}  We have

$$H^*(\Gamma,E^{a,b}_{r,s}) \simeq H^*(\Gamma,E_{r,s})^{(a,b)}$$
as Hecke modules.
\end{lemma}

\section{Induced Modules}

As we announced in the introduction we want to go down to trivial weight by increasing the level by $\ell$. Thus we are interested in the 
Hecke module $H^1(\Gamma(\ell),\F_{\ell})$. We investigate these in this section.

Let $\chi: \Gamma^0(\ell) / \Gamma(\ell) \rightarrow \F_{\ell}^*$ be a homomorphism. For any $\F_{\ell}[\Delta]$-module $E$, we define
$H^*(\Gamma(\ell),\chi,E)$ as the submodule of all $v \in H^*(\Gamma(\ell),E)$ such that $v\cdot \bigl( \begin{smallmatrix} a & b  \\ c & d \\ \end{smallmatrix} \bigr) = v \cdot \chi(d)$ for every $\bigl( \begin{smallmatrix} a & b  \\ c & d \\ \end{smallmatrix} \bigr) \in \Gamma^0(\ell)$.

We have

$$H^1(\Gamma(\ell),\F_{\ell}) \simeq \bigoplus_{\chi}H^1(\Gamma(\ell),\chi,\F_{\ell}) \simeq  \bigoplus_{\chi}H^1(\Gamma^0(\ell), (\F_{\ell})^{\chi})$$

\noindent where $(\F_{\ell})^{\chi}$ is the rank one $\F_{\ell}$-module on which $\Gamma^0(\ell)$ acts via $\chi$. The last isomorphism follows from Lemma 1.1.5 of \cite{as1}. Using Shapiro's lemma, we relate these to the cohomology of $\Gamma$.

$$H^1(\Gamma(\ell),\F_{\ell}) \simeq \bigoplus_{\chi}H^1(\Gamma,\text{Ind}(\Gamma^0(\ell),\Gamma, (\F_{\ell})^{\chi})).$$

We follow Ash and Stevens and use the following space of functions in order to study the module $\text{Ind}(\Gamma^0(\ell),\Gamma, (\F_{\ell})^{\chi})$. Let $I$ be the set of $\F_{\ell}$ valued functions on $\F_{\ell}^2$ which vanish at the origin. The semigroup $\Delta$ acts on $I$ both by reduction by $\lambda$ and by $\bar{\lambda}$. The action is given by $$(f \cdot M)(a,b) = f((a,b)M^t)$$ for $f \in I$, $(a,b) \in \F_{\ell}^2$ and $M \in \Delta$.
  
 For each integer $n$, let $I_n$ be the $\Delta$-submodule of $I$ consisting of homogeneous functions of degree $n$, that is, the collection of functions $f \in I$ such that $f((x a, x b))=x^n f((a,b))$. Observe that $I_k=I_{k+l-1}$. A function $f \in I_n$ is determined by its values on the set $\{ (1,0),...,(1,\ell-1),(0,1)\}$ , which can be identified with $\mathbb{P}^1(\F_{\ell})$. Thus every $I_n$ is $\ell+1$ dimensional. We have the decomposition $$I \simeq \bigoplus_{n=0}^{\ell-2} I_n.$$
 
Let $\chi_1 : (\O / \lambda)^* \rightarrow \F_{\ell}^*$ and $\chi_2 : (\O / \bar{\lambda})^* \rightarrow \F_{\ell}^*$ be the restrictions of the canonical isomorphisms to the units. We have the following isomorphisms of $\Delta$-modules
$$\text{Ind}(\Gamma^0(\lambda),\Gamma,(\F_{\ell})^{\chi_1^k}) \simeq I_k $$ 
and
$$\text{Ind}(\Gamma^0(\bar{\lambda}),\Gamma,(\F_{\ell})^{\chi_2^k}) \simeq I_k $$ 
for $0 \leq k \leq l-2$. Of course, in the first case $\Delta$ acts on $I_k$ via reduction through $\lambda$ and in the second case via reduction through $\bar{\lambda}$.

 As the quotient $\Gamma^0(\ell) / \Gamma(\ell)$ is isomorphic to $(\O/\ell)^* \simeq (\O / \lambda)^* \times (\O / \bar{\lambda})^*$, any homomorphism $\chi: \Gamma^0(\ell) / \Gamma(\ell) \rightarrow \F_{\ell}^*$ can be written uniquely as a product $\chi_1^r \cdot \chi_2^s$ for some $0 \leq r,s \leq l-1$. In this case, we denote $\chi$ as $\chi(r,s)$.

The following is a straightforward generalization of Lemma 2.6 of \cite{as2}.
\begin{lemma}\label{lemma-induced} Let $0 \leq r,s, \leq l-1$. Then

$$H^1(\Gamma, I_r \otimes_{\F_{\ell}} I_s) \simeq H^1(\Gamma(\ell), \chi(r,s),\F_{\ell})$$
as Hecke modules.

\end{lemma}

\begin{proof} As before, we have 
$$H^1(\Gamma(\ell), \chi(r,s),\F_{\ell})) \simeq H^1(\Gamma^0(\ell), (\F_{\ell})^{\chi(r,s)}) \simeq H^1(\Gamma,\text{Ind}(\Gamma^0(\ell),\Gamma, (\F_{\ell})^{\chi(r,s)}) ).$$
So it suffices to show that

  $$\text{Ind}(\Gamma^0(\ell),\Gamma, (\F_{\ell})^{\chi(r,s)}) \simeq \text{Ind}(\Gamma^0(\lambda),\Gamma, (\F_{\ell})^{\chi_1^r})\otimes \text{Ind}(\Gamma^0(\bar{\lambda}),\Gamma, (\F_{\ell})^{\chi_2^s}).$$

\noindent Observe that $P(\ell)$, the intersection of the principal congruence subgroup of level $\ell$ and $\Gamma$, acts trivially on $(\F_{\ell})^{\chi(r,s)}$. Thus after factoring, we get 
$$\text{Ind}(\Gamma^0(\ell),\Gamma, (\F_{\ell})^{\chi(r,s)}) \simeq \text{Ind}(\text{B}(\O / \ell),\text{SL}_2(\O / \ell), (\F_{\ell})^{\chi(r,s)})$$
\noindent where $\text{B}(\O / \ell)$ is the subgroup of upper triangular matrices in $\text{SL}_2(\O / \ell)$. Notice that we have 
$\text{B}(\O / \ell) \simeq \text{B}(\O / \lambda) \times \text{B}(\O / \bar{\lambda})$ and 
$\text{SL}_2(\O / \ell) \simeq \text{SL}_2(\O / \lambda) \times \text{SL}_2(\O / \bar{\lambda})$. This gives 
$$\text{Ind}(\text{B}(\O / \ell),\text{SL}_2(\O / \ell), (\F_{\ell})^{\chi(r,s)}) \simeq \text{Ind}(\text{B}(\O / \lambda),\text{SL}_2(\O / \lambda), (\F_{\ell})^{\chi_1^r}) \otimes \text{Ind}(\text{B}(\O / \bar{\lambda}),\text{SL}_2(\O / \bar{\lambda}), (\F_{\ell})^{\chi_2^s}).$$
The claim follows easily from here.
\end{proof}

 \section{Exact Sequences}
 We will need the following two facts, see \cite[Section 3]{as2}. 
  
\begin{lemma}\label{lemma-pairings} 
 For $0 \leq r \leq \ell-1$, there are $SL_2(\O)$-invariant perfect pairings
 \begin{enumerate} 
  \item[(1)] $E_r \times E_r \rightarrow \F_{\ell}$
  \item[(2)] $I_r \times I_{\ell-1-r} \rightarrow \F_{\ell}$
 \end{enumerate}
\end{lemma}

 Let $0 \leq r \leq \ell-1$. As in \cite{as2}, we consider the following $SL_2(\O)$-invariant maps. Each polynomial in $E_r$ can be seen as a function on $\F_{\ell}^2$. This gives us a morphism $\alpha_r : E_g\rightarrow I_r$. Let $\beta_r : I_r \rightarrow E_{\ell-1-r}^r$ be given by $$\beta_r(f)=\displaystyle{\sum_{(a,b) \in \F_{\ell}^2}}f(a,b)(bX-aY)^{\ell-1-r}.$$ 
 
\begin{lemma}\label{lemma-exsqs}
 For $0 \leq r \leq \ell-1$, we have the following exact sequence of $\Delta$-modules
 $$\xymatrix{ 0 \ar[r] & E_r \ar[r]^{\alpha_r} & I_r \ar[r]^{\beta_r} & E_{\ell-1-r}^r \ar[r] & 0  }$$
\end{lemma}
  
 \begin{remark} Lemma \ref{lemma-exsqs} shows that the semisimplification of $I_r$ is $E_r \oplus E_{\ell-1-r}^r$. There is another way to see this. As we explained in the proof of Lemma \ref{lemma-induced}, $I_r$ is the induction of the one dimensional representation $\chi^r$ of the Borel subgroup of $\text{SL}_2(\F_{\ell})$ to all of $\text{SL}_2(\F_{\ell})$. One can identify the semisimplification of this $\ell+1$ dimensional representation by a calculation of Brauer characters. This has been done by Diamond in \cite[Prop 1.1.]{dia}.
\end{remark}
 
 \begin{definition}
 For given nonnegative integers $r,s$, we define the following $\Delta$-modules where $\Delta$ acts on the components of every tensor product through reduction by $\lambda$ and $\bar{\lambda}$ respectively.
\begin{enumerate}
 
 \item[1.] $I_{r,s}:= I_r\otimes I_s$;

 \item[2.] $U_{r,s}:= [E_{\ell-1-r}^r\otimes I_s] \oplus [I_r\otimes
E_{\ell-1-s}^s]$;

 \item[3.] $V_{r,s}:= E_{\ell-1-r}^r \otimes E_{\ell-1-s}^s$.
\end{enumerate}
\end{definition}

 We have $\Delta$-module morphisms $$\pi:I_{r,s}\rightarrow
U_{r,s}\indent \text{ defined by }\indent \pi:=[\beta_r\otimes \mathrm{id}]\oplus [\mathrm{id}\otimes
\beta_s]$$ and $$\pi': U_{r,s}\rightarrow V_{r,s}\indent\text{ defined by }\indent
\pi':=\mathrm{id}\otimes \beta_s - \beta_r\otimes \mathrm{id}.$$

\begin{lemma}\label{lemma-exsq}
 Let the notation be as above. Let $0\leq r\leq \ell-1 $ and $0\leq s\leq \ell-1$. We have the following exact sequence $\Delta$-modules:
 $$\xymatrix{0\ar[r] & E_{r,s}\ar[r]^\iota & I_{r,s}\ar[r]^\pi & U_{r,s}\ar[r]^{\pi'} & V_{r,s}\ar[r] & 0}.$$ 
\end{lemma}

\begin{proof}
 Note that $\Delta$-modules in question are flat since they are also $\F_\ell$-vector spaces. So, by Lemma \ref{lemma-exsqs}, $\iota$ is injective. One can easily see that $Im(\iota)\subseteq Ker(\pi)$ and $\pi'$ is surjective. Thus, in order to complete the proof, it suffices to show that $dim(Im(\pi))=(\ell+1)^2-(r+1)(s+1)$; this is what we do below. 

 Identifying $E_r$ with its image in $I_r$, we can write the vector space decomposition $I_r=E_r \oplus E_{\ell-1-r}$. Now, it is evident that $dim(\pi(E_r\otimes I_s))=(r+1)(\ell-s)$ and that $dim(\pi(E_{\ell-1-r}\otimes I_s))=(\ell-r)(\ell+1)$. Elementary linear algebra shows that these images have trivial intersection and this gives us the desired dimension.
\end{proof}

 Setting $W_{r,s}:=ker(\pi':U_{r,s}\rightarrow V_{r,s})$, by Lemma \ref{lemma-exsq}, we get two short exact sequences
 
\begin{equation}\label{exsq1}
\xymatrix{0\ar[r] & E_{r,s}\ar[r]^\iota & I_{r,s}\ar[r]^\pi & W_{r,s}\ar[r] & 0}
\end{equation}
\noindent and
\begin{equation}\label{exsq2}
\xymatrix{0\ar[r] & W_{r,s}\ar[r]^i & U_{r,s}\ar[r]^{\pi'} & V_{r,s}\ar[r] & 0}.
\end{equation}


\section{Invariants}

For convenience, we will write $\F_{\ell}(g)$ for the module $E^{g,g}_{0,0}$ which we defined in Section 3.

\begin{lemma}\label{prop-I}
 For any nonnegative integers $r,s$, we we have the following isomorphism of Hecke modules
 \begin{center}
  \begin{tabular}{rrl}
   $H^0(\Gamma,I_{r,s})$ &$\cong \Bigl \lbrace$& \begin{tabular}{ll} $\F_{\ell}(\ell-1)$ & if $r\equiv s\equiv 0$  (mod $\ell -1$)  \\ 0 & otherwise \end{tabular} \\
  \end{tabular}
 \end{center}
\end{lemma} 

\begin{proof}
By Shapiro's Lemma, we have $H^0(\Gamma,I_{r,s})\simeq H^0(\Gamma^0(\ell), (\F_{\ell})^{\chi(r,s)}).$ In action of $\Gamma^0(\ell)$ on $\F_{ell}$ through $\chi(r,s)$, either there are no nontrivial invariants or the whole space is fixed which means that $\chi(r,s)$ acts trivially. By the Chinese Remainder Theorem, this is possible if and only if $\chi_1^r$ and $\chi_2^s$ act trivially. Hence the congruence conditon of the claim. One can directly check that the Hecke action is as described.
\end{proof}

\begin{lemma}\label{lemma-E}
Assume $0\leq r,s\leq \ell-1$. Then, we have the following isomorphism of Hecke modules
 \begin{center}
  \begin{tabular}{rrl}
   $H^0(\Gamma,E_{r,s})$ &$= \Bigl \lbrace$& \begin{tabular}{ll} $\F_{\ell}$ & if $r=s= 0$  \\ 0 & otherwise \end{tabular} \\
  \end{tabular}
 \end{center} 
\end{lemma}

\begin{proof}
 The claim is obvious when $(r,s)=(0,0)$. Assume $(r,s)\neq (0,0)$ and $(r,s)\neq (\ell-1,\ell-1)$. Then, the exact sequence (\ref{exsq1}) induces the following exact sequence $$0\rightarrow H^0(\Gamma,E_{r,s})\rightarrow H^0(\Gamma,I_{r,s}).$$ By Proposition \ref{prop-I},  $H^0(\Gamma,I_{r,s})=0$ and so is $ H^0(\Gamma,E_{r,s})$. 

 Assume $(r,s)=(\ell-1,\ell-1)$. We have the isomorphism $E_{\ell-1, \ell-1}\cong (\O/\ell)[x,y]_{\ell-1}$. On the other hand, in \cite{d}, Dickson showed that $\Gamma$ invariants of $\tilde{E}_{*}$ is generated by $X^\ell Y-XY^\ell$and $\sum_{i=0}^\ell (X^{\ell-i}Y^i)^{\ell-1}$. This implies that $H^0(\Gamma, E_{\ell-1,\ell-1})=0$.    
\end{proof}

\begin{lemma}\label{lemma-U}
 Let $0\leq r,s\leq \ell-1$. Then, we have the following isomorphism of Hecke modules
 \begin{center}
  \begin{tabular}{rrl}
   $H^0(\Gamma,U_{r,s})$ &$= \Bigl \lbrace$& \begin{tabular}{ll} $\F_{\ell}(\ell-1)\oplus \F_\ell(\ell-1)$ & if $(r,s)= (\ell-1,\ell-1)$  \\ $\F_{\ell}(\ell-1)$ & if $(r,s)= (0,\ell-1)$ or $(\ell-1,0)$ \\ 0 & otherwise \end{tabular} \\
  \end{tabular}
 \end{center} 
\end{lemma}

\begin{proof}
 Set $U^1:=E_{\ell-1-r}(r)\otimes I_s$ and $U^2=I_r \otimes E_{\ell-1-s}(s)$. Then, $U_{r,s}=U^1\oplus U^2$ and $H^0(\Gamma, U^1)\oplus H^0(\Gamma, U^2)$. 
 
  Assume $(r,s)$ is not of $(\ell-1,\ell-1)$, $(0,\ell-1)$ and $(\ell-1,0)$. Then, tensoring the exact sequence in Lemma \ref{lemma-exsqs} with $E_{\ell-1-r}(r)$, we get the following short exact sequence $$\xymatrix{0\ar[r] & E_{\ell-1-r}(r)\otimes E_s\ar[r] & U^1\ar[r] & V_{r,s}\ar[r] & 0}.$$ This induces the following long exact sequence $$\xymatrix{0\ar[r] & H^0(\Gamma,E_{\ell-1-r}(r)\otimes E_s)\ar[r] & H^0(\Gamma, U^1)\ar[r] & H^0(\Gamma,V_{r,s}). }$$ Since $V_{r,s}\cong E_{\ell-1-r,\ell-1-s}$ as $\Gamma$-modules, by Lemma \ref{lemma-E}, $H^0(\Gamma,V_{r,s})=0$. On the other hand, by Lemma \ref{lemma-E}, $H^0(\Gamma,E_{\ell-1-r}(r)\otimes E_s)=0$ and $H^0(\Gamma,U^1)=0$. Likewise, one tensors the exact sequence in Lemma \ref{lemma-exsqs} with $E_{\ell-1-s}(s)$ and gets $H^0(\Gamma,U^2)=0$, hence the vanishing of $H^0(\Gamma,U_{r,s})$. 
  
  Now, assume $(r,s)=(\ell-1,0)$. Then, by Lemma \ref{lemma-E}, $H^0(\Gamma,E_{\ell-1-r}(r)\otimes E_s)\cong \F_\ell$ and $H^0(\Gamma,V_{r,s})=0$. Using the exact sequence of cohomology groups above, we conclude that $H^0(\Gamma, U^1)\cong \F_\ell$ as vector spaces. Likewise, one gets $H^0(\Gamma,U^2)=0$.
  
  In case $(r,s)=(0,\ell-1)$, one proceeds exactly as above and gets $H^0(\Gamma,U^1)=0$ and  $H^0(\Gamma,U^2)=\F_\ell$.
  
  Finally assume $(r,s)=(\ell-1,\ell-1)$. In this case, $H^0(\Gamma,E_{\ell-1-r}(r)\otimes \bar{E}_s)=0$ and $H^0(\Gamma, V_{r,s})\cong \F_{\ell}$ by Lemma \ref{lemma-E}. One can easily see that $\pi'|_{U^1}:U^1\rightarrow V_{r,s}$ is surjective and so $H^0(\Gamma,U^1)\cong \F_\ell$ as vector spaces. Exactly in the same way, one gets $H^0(\Gamma,U^2)\cong \F_\ell$ (as vector spaces). One checks the action of the Hecke algebra and completes the proof.
\end{proof}

\begin{lemma}\label{lemma-W}
 Let $0\leq r,s\leq \ell-1$. Then, we we have the following isomorphism of Hecke modules
 \begin{center}
  \begin{tabular}{rrl}
   $H^0(\Gamma,W_{r,s})$ &$= \Bigl \lbrace$& \begin{tabular}{ll} $\F_{\ell}$ & if $(r,s)=(\ell-1,\ell-1), (0,\ell-1)$ or $(\ell-1,0)$  \\ 0 & otherwise \end{tabular} \\
  \end{tabular}
 \end{center} 
\end{lemma}

\begin{proof}
 First of all, the exact sequence (\ref{exsq2}) above induces the following long exact sequence of Hecke modules in cohomology $$\xymatrix{0\ar[r] & H^0(\Gamma,W_{r,s})\ar[r]^{i_*} & H^0(\Gamma,U_{r,s})\ar[r]^{\pi_*'} & H^0(\Gamma, V_{r,s})\ar[r] & H^1(\Gamma,W_{r,s})}.$$ 
 
  Assume $(r,s)=(0,\ell-1)$ or $(\ell-1,0)$. Then, by Lemma \ref{lemma-E}, $H^0(\Gamma,V_{r,s})=0$. The proof immediately follows from Lemma \ref{lemma-U}. 
  
  Assume $(r,s)=(\ell-1,\ell-1).$ Then, by Lemma \ref{lemma-E}, $H^0(\Gamma,V_{r,s})\cong \F_\ell$ and, by Lemma \ref{lemma-U}, $H^0(\Gamma,U_{r,s})\cong \F_\ell\oplus \F_\ell$. Using the definition, one can easily see that $\pi'_*$ is surjective and gets the desired result using the exact sequence of cohomology groups above.

 Finally, assume $(r,s)$ is not equal to one of $(0,\ell-1)$, $(\ell-1,0)$ and $(\ell-1,\ell-1)$. Then, by Lemma \ref{lemma-U}, $H^0(\Gamma,U_{r,s})=0$ and, using the exact sequence above, we complete the proof. 
\end{proof}

 \begin{remark} One can compute the above invariants using the following approach which was suggested by Gebhard Boeckle. As $P(\ell)$ acts trivially on $E_{r,s}$, $I_{r,s}$ and the direct summands of $U_{r,s}$, we get, for instance, $H^0(\Gamma, E_{r,s}) \simeq H^0(\text{SL}_2(\O / \lambda)\times \text{SL}_2(\O / \bar{\lambda}),E_{r,s})$. Since we are taking invariants, we get 
  $$H^0(\Gamma / (P(\ell) \cap \Gamma),E_{r,s}) \simeq H^0(\text{SL}_2(\O / \lambda), E_r) \otimes H^0(\text{SL}_2(\O / \bar{\lambda}),E_s).$$
 This gives
 $$H^0(\Gamma,E_{r,s}) \simeq H^0(\Gamma, E_r) \otimes H^0(\Gamma,E_s).$$
 Now one can follow the proof of Lemma 3.3 of \cite{as2} to compute these invariants. Same approach applies to Lemmas \ref{lemma-U} and \ref{lemma-W} as well.

 \end{remark}
 
 \section{Proof Of The Theorem}
 We are now ready to prove our main result:

\begin{theorem}\label{thm_main}
 Let $\Phi$ be a Hecke eigenvalue system occuring in $H^1(\Gamma,E^{a,b}_{r,s})$ for some $0 \leq a,b \leq l-2$ and $0 \leq r,s \leq l-1$. Then $\Phi$ occurs in $H^1(\Gamma_2,\chi(r,s),\F_{\ell})^{(a,b)}$.
\end{theorem}

\begin{proof}
 By Lemma \ref{lemma-det}, we have $H^1(\Gamma,E^{a,b}_{r,s}) \simeq H^1(\Gamma,E_{r,s})^{(a,b)}$.
 Exact sequence (\ref{exsq1}) induces the following long exact sequence of $\mathbb{H}$-modules 
\[ 0 \rightarrow   H^0(\Gamma,E_{r,s})^{(a,b)}\xrightarrow{\iota_*}  H^0(\Gamma,I_{r,s})^{(a,b)} \xrightarrow{\pi_*'}  H^0(\Gamma, W_{r,s})^{(a,b)} \rightarrow  H^1(\Gamma,E_{r,s})^{(a,b)} \rightarrow  H^1(\Gamma,I_{r,s})^{(a,b)} \] 

We claim that the map $H^1(\Gamma,E_{r,s})\rightarrow H^1(\Gamma,I_{r,s})$ is injective. Assume that $(r,s)$ is equal to one of the tuples $(0,\ell-1),(\ell-1,0)$ or $(\ell-1,\ell-1)$. Then, by Lemma \ref{lemma-E}, $H^0(\Gamma,E_{r,s})=0$; by Lemma \ref{prop-I}, $H^0(\Gamma,I_{r,s})\cong \F_\ell$ and, by Lemma \ref{lemma-W}, $H^0(\Gamma,W_{r,s})\cong \F_\ell$ (as vector spaces). By the definition, $\pi'_*$ is surjective and thus we get the claim. Otherwise, by Lemma \ref{lemma-W}, $H^0(\Gamma,W_{r,s})=0$.
 
 Now the result follows from Proposition \ref{lemma-induced}
\end{proof}

 Let $\Phi$ be a Hecke eigenvalue system occuring in $H^1(\Gamma,E)$ where $E$ is some rational finite dimensional $\F_{\ell}[\text{GL}_2(\O / (\ell))]$-module. Then Lemma \ref{lemma-irred} tells us that $\Phi$ can be realized in some $H^1(\Gamma,E^{a,b}_{r,s})$ with $0 \leq a,b \leq l-2$ and $0 \leq r,s \leq l-1$. Thus our main theorem implies the followings as we announced in the introduction.
 
 \begin{corollary} Let $\Phi$ be a Hecke eigenvalue system occuring in $H^1(\Gamma_1(\a),E)$ where
 $E$ is a finite dimensional $\F_{\ell}[\text{GL}_2(\O / (\ell))]$-module. Then $\Phi$ occurs in $H^1(\Gamma_1(\a \ell),\F_{\ell})$, up to determinant twist. 
 \end{corollary}
 
 For congurence subgroups of $SL_2(\mathbb{Z})$, the following was first proved by Tate-Serre for level $1$ (unpublished), by Jochnowitz \cite{jo} for prime levels less than $19$ and for arbitrary levels by Ash-Stevens \cite{as2}.
 
 \begin{corollary}
 The set of Hecke eigenvalue systems occuring in $H^1(\Gamma_1(\a),E)$ for fixed $\a$ and varying $E$, where
 $E$ is a rational finite dimensional $\F_{\ell}[\text{GL}_2(\O / (\ell))]$-module, is finite.
 \end{corollary}

It is natural to ask whether increasing the level by $(\ell)$ is
optimal. In other words, are there eigenvalue systems with nontrivial
weight which have no twists that occur with trivial weight when the level is increased
by $(\lambda)$ or $(\bar{\lambda})$. One can see, by the methods we used in this note, that the answer to this question is positive if $r=0$ or $s=0$. It looks like this is the only case where the answer is positive. We present a numerical example to support this speculation. 

\begin{example}
  Let $\O=\Z(\omega)$ where $\omega=\sqrt{-2}$. Using the programs
developed by the first author in his doctoral thesis \cite{sen}, we
find an eigenform $v$ in $H^1(\text{PSL}_2(\O), E_{10,10}(\F_{11}))$. The
following table gives eigenvalues $\Phi_{\alpha}$ of $v$ for the first few
Hecke operators $T_{\alpha}$.

  \begin{table}[h]
  \centering
  \begin{tabular}{|c|c|c|c|c|c|c|c|c|c|} \hline

$\alpha$&$\omega$&$1+\omega$&$1-\omega$&$3+2\omega$&$3-2\omega$&$1+3\omega$&$1-3\omega$&$3+4\omega$&$3-4\omega$
\\
\hline
  $\Phi_{\alpha}$&9&10&10&9&9&0&0&5&5 \\ \hline
  \end{tabular}
  \end{table}

\noindent Note that we have $11=(3+\omega)(3-\omega)$. The spaces
$H^1(\Gamma_0(3+\omega),\F_{11})$ and
$H^1(\Gamma_0(3-\omega),\F_{11})$ are
isomorphic and they are two dimensional. Our eigenvalue system $\Phi$
does not occur in these spaces. Next, we examine
$H^1(\Gamma_0(11),\F_{11})$. We find an eigenvalue system that is the reduction of a characteristic 0 system. Indeed, we find an eigenvector in $H^1(\Gamma_0(11),\O)$ with the following
eigenvalues $\Psi_{\alpha}$. 

\begin{table}[h]
  \centering
  \begin{tabular}{|c|c|c|c|c|c|c|c|c|c|} \hline

$\alpha$&$\omega$&$1+\omega$&$1-\omega$&$3+2\omega$&$3-2\omega$&$1+3\omega$&$1-3\omega$&$3+4\omega$&$3-4\omega$
\\
\hline
  $\Psi_{\alpha}$&-2&-1&-1&-2&-2&0&0&-6&-6 \\ \hline
  \end{tabular}
  \end{table}

  \noindent Reducing these eigenvalues mod $11$, we get an eigenvalue
system in $H^1(\Gamma_0(11),\F_{11})$ that matches (we computed only the first 20 split primes) our level 1 weight
$(10,10)$ eigenvalue system $\Phi$.

\end{example}


   \vspace{.2 in}
   \textsc{department of mathematics, university of duisburg-essen, 2 Universitatstrasse Essen Germany}
   
   \textit{E-mail address:} mehmet.sengun@uni-due.de\\
   
   \textsc{department of mathematics, university of wisconsin, 480 Lincoln Dr Madison wi 53706}
   
   \textit{E-mail address:} turkelli@math.wisc.edu


\begin{thebibliography}{100}

 \bibitem[AS1]{as1} A.Ash, G.Stevens ; Cohomology of arithmetic groups and congruences between systems of Hecke eigenvalues.  { \it J. Reine Angew. Math.}  { \bf 365}  (1986), 192--220. 

 \bibitem[AS2]{as2} A.Ash, G.Stevens ; Modular forms in characteristic $\ell$ and special values of their $L$-functions.  { \it Duke Math. J.}  {\bf 53}  (1986),  no. 3, 849--868. 
 
 \bibitem[BH]{bh} T. Berger, G. Harcos ; $\ell$-adic representations associated to modular forms over imaginary quadratic fields, Int. Math. Res. Not., 2007, no.23
 
 \bibitem[BN]{bn} R. Brauer and C. Nesbitt, On the modular characters of groups, Ann. of Math. (2), 42 (1941), pp. 556-590. 

 \bibitem[DI]{di} F.Diamond, J.Im ; Modular forms and modular curves.  Seminar on Fermat's Last Theorem 39--133, CMS Conf. Proc., 17, Amer. Math. Soc., Providence, RI, 1995
 
 \bibitem[Di]{dia} F.Diamond ; A correspondence between representations of local Galois groups and Lie-type groups.  $L$-functions and Galois representations,  187--206, London Math. Soc. Lecture Note Ser., 320, Cambridge Univ. Press, Cambridge, 2007.
 
 \bibitem[Ds]{d} L.E. Dickson; A fundamental system of invariants of the general modular linear group with solution of the form problem, Trans. Amer. of Math. Soc., 12, (1911) 75-98.
  
 \bibitem[EGM]{egm} J.Elstrod, F.Grunewald, J.Mennicke ; ${\rm PSL}(2)$ over imaginary quadratic integers.  Arithmetic Conference (Metz, 1981),  43--60, Astérisque, 94, Soc. Math. France, Paris, 1982.
  
  \bibitem[Fi]{fi} L.M.Figueiredo ; Serre's conjecture for imaginary quadratic fields. {\it Compositio Math.} {\bf 118} (1999), no. 1, 103--122
 
 \bibitem[Fr]{fr} J.Franke ; Harmonic analysis in weighted $L\sb 2$-spaces. {\it Ann. Sci. École Norm. Sup.} (4)  {\bf 31}  (1998),  no. 2, 181--279. 
 
 \bibitem[HST]{hst} M.Harris, D.Soudry, R.Taylor ; $\ell$-adic representations associated to modular forms over imaginary quadratic fields. I. Lifting to ${\rm GSp}\sb 4(Q)$.  { \it Invent. Math.}  { \bf 112}  (1993),  no. 2, 377--411. 
 
 \bibitem[J]{jo} N. Jochnowitz ; Congurences between systems of eigenvalues of modular forms, Trans. Amer. math. Soc. 270 (1982), 269-285.
 
 \bibitem[Sen]{sen} M. H. Sengun ; Serre's conejcture over imaginary quadratic fields, PhD thesis, UW-Madison, 2008
 
 \bibitem[Sh]{sh} G.Shimura ; Introduction to the arithmetic theory of automorphic functions. Publications of the Mathematical Society of Japan, 11. Kanô Memorial Lectures, 1. Princeton University Press, Princeton, NJ, 1994. 
 
 \bibitem[St]{st} W.Stein ; Modular forms, a computational approach. With an appendix by Paul E. Gunnells. Graduate Studies in Mathematics, 79. American Mathematical Society, Providence, RI, 2007.
 
 \bibitem[Ta1]{ta1} R.Taylor ; $\ell$-adic representations associated to modular forms over imaginary quadratic fields. II. {\it Invent. Math.} 
{\bf 116}  (1994),  no. 1-3, 619--643.

 \bibitem[Ta2]{ta2} R. Taylor;  On congruences between modular forms. PhD. Thesis, Princeton University 1998. 
 
 \bibitem[Wi]{wi} G.Wiese ; On the faithfulness of parabolic cohomology as a Hecke module over a finite field.  {\it J. Reine Angew. Math.}  
{\bf f606}  (2007), 79--103
 
\end{thebibliography}
  \end{document}